\documentclass[12pt,a4paper]{article}

\usepackage{fullpage}

\usepackage{amsmath}
\usepackage{amssymb}
\usepackage{amscd}
\usepackage{amsfonts}
\usepackage{amsthm}
\usepackage{tikz}
\usetikzlibrary{calc}

\font\smallit=cmti10

\theoremstyle{plain}
\newtheorem{thm}{Theorem}
\newtheorem{lem}{Lemma}

\theoremstyle{definition}
\newtheorem{defn}{Definition}

\newtheorem*{problem}{Problem}

\theoremstyle{remark}

\newtheorem{notation}{Notation}

\newcommand{\Z}[1]{\mathbb{Z} /#1 \mathbb{Z}}
\newcommand{\m}{\mathfrak{m}}

\newcommand{\N}{\mathbb{N}}

\newcounter{x}

\begin{document}
\begin{center}
\uppercase{\bf On the problem of Molluzzo for the modulus 4}
\vskip 20pt
{\bf Jonathan Chappelon}\\
{\smallit I3M, Universit\'{e} Montpellier 2, CC 51, Place Eug\`{e}ne Bataillon, 34095 Montpellier Cedex 5, France.}\\
{\tt jonathan.chappelon@math.univ-montp2.fr}\\ 
\vskip 10pt
{\bf Shalom Eliahou}\\
{\smallit \begin{tabular}{l}
Univ Lille Nord de France, F-59000 Lille, France, \\
ULCO, LMPA J. Liouville, B.P. 699, F-62228 Calais, France, \\
CNRS, FR 2956, France. \\
\end{tabular}}\\
{\tt shalom.eliahou@lmpa.univ-littoral.fr}\\
\vskip 10pt
December 17, 2011
\end{center}

\begin{abstract}
We solve the currently smallest open case in the 1976 problem of Molluzzo on $\Z{m}$, namely the case $m=4$. This amounts to constructing, for all positive integer $n$ congruent to 0 or 7 mod 8, a sequence of integers modulo 4 of length $n$ generating, by Pascal's rule, a Steinhaus triangle containing 0,1,2,3 with equal multiplicities.
\end{abstract}

\section{Introduction} The problem of Molluzzo in combinatorial number theory is about the existence of certain triangular arrays in $\Z{m}$. It was first formulated by Steinhaus in 1958 for $m=2$ \cite{St}, and then generalized by Molluzzo in 1976 for $m \ge 3$ \cite{Mo}. It is still widely open for most moduli $m$. The problem reads as follows. Given $m \ge 2$, for which $n \ge 1$ does there exist a triangle 
\begin{center}
\begin{tikzpicture}[scale=0.5]

\pgfmathparse{sqrt(3)}\let\x\pgfmathresult

\node at (0,3.5*\x) {$x_{1,1}$};
\node (a) at (2,3.5*\x) {$x_{1,2}$};
\node (c) at (1,3*\x) {$x_{2,1}$};

\node at (8,3.5*\x) {$x_{1,n}$};
\node (b) at (6,3.5*\x) {$x_{1,n-1}$};
\node (e) at (7,3*\x) {$x_{2,n-1}$};

\node (d) at (4,0) {$x_{n,1}$};

\draw[dotted,line width=1pt] (a) -- (b);
\draw[dotted,line width=1pt] (c) -- (d);
\draw[dotted,line width=1pt] (e) -- (d);

\node at (-1.5,2*\x) {$\nabla =$};

\end{tikzpicture}
\end{center}
with entries $x_{i,j}$ in $\Z{m}$, of side length $n$, satisfying the following two conditions:
\begin{enumerate}
\item[(C1)] Pascal's rule: every element of $\nabla$ outside the first row is the sum of the two elements above it. That is, $x_{i,j}=x_{i-1,j}+x_{i-1,j-1}$ for all $2\leq i\leq n$ and $1\leq j\leq n-i$.
\item[(C2)] The elements of $\Z{m}$ all occur with the same multiplicity in $\nabla$.
\end{enumerate}
\begin{defn} 
A triangle $\nabla$ in $\Z{m}$ satisfying condition (C1) is called a Steinhaus triangle. It is said to be {\em balanced} if it satisfies condition (C2).
\end{defn}
By (C1), a Steinhaus triangle is completely determined by its first row $S$, and may thus be denoted by $\nabla S$. For instance, the sequence $S= 0100203$ in $\Z{4}$ generates the following Steinhaus triangle $\nabla S$: 
\begin{center}
\begin{tikzpicture}[scale=0.25]
\pgfmathparse{sqrt(3)}\let\x\pgfmathresult

\node at (0,6*\x) {$0$};
\node at (2,6*\x) {$1$};
\node at (4,6*\x) {$0$};
\node at (6,6*\x) {$0$};
\node at (8,6*\x) {$2$};
\node at (10,6*\x) {$0$};
\node at (12,6*\x) {$3$};

\node at (1,5*\x) {$1$};
\node at (3,5*\x) {$1$};
\node at (5,5*\x) {$0$};
\node at (7,5*\x) {$2$};
\node at (9,5*\x) {$2$};
\node at (11,5*\x) {$3$};

\node at (2,4*\x) {$2$};
\node at (4,4*\x) {$1$};
\node at (6,4*\x) {$2$};
\node at (8,4*\x) {$0$};
\node at (10,4*\x) {$1$};

\node at (3,3*\x) {$3$};
\node at (5,3*\x) {$3$};
\node at (7,3*\x) {$2$};
\node at (9,3*\x) {$1$};

\node at (4,2*\x) {$2$};
\node at (6,2*\x) {$1$};
\node at (8,2*\x) {$3$};

\node at (5,\x) {$3$};
\node at (7,\x) {$0$};

\node at (6,0) {$3$};
\end{tikzpicture}
\end{center}

Observe that $\nabla S$ is balanced, since each element of $\Z{4}$ appears with the same multiplicity, namely 7.

\smallskip

An obvious necessary condition for the existence of a balanced Steinhaus triangle $\nabla$ in $\Z{m}$ of side length $n$ is given by
\begin{equation}\label{nec}
{n+1 \choose 2} \, \equiv \, 0 \bmod m.
\end{equation}
Indeed, it follows from (C2) that $m$ divides the multiset cardinality of $\nabla$, which is ${n+1 \choose 2}$.

\smallskip

Is this necessary condition also sufficient? This is the more detailed content of Molluzzo's problem. While it seems reasonable to conjecture a positive answer in most cases, two counterexamples are known: in $\Z{15}$ and in $\Z{21}$, there is no balanced Steinhaus triangle of side length 5 and 6, respectively, even though the pairs $(m,n)=(15,5)$ and $(21,6)$ both satisfy condition \eqref{nec}. (See \cite[p. 75]{Ch0} and \cite[p. 293]{Ch2}.)

\smallskip

Note that, given $m \ge 2$, the condition for $n \in \N$ to satisfy \eqref{nec} only depends on the class of $n$ mod $m$ if $m$ is odd, or mod $2m$ if $m$ is even. 

\subsection{Known results}\label{known}

Despite its apparent simplicity, the problem of Molluzzo is very challenging, as testified by the scarcity of available results. The only moduli $m$ for which it has been completely solved so far are 
\begin{itemize}
\item $m = 2$ in \cite{Ha,EH,EMR}, 
\item $m=3^i$ for all $i \ge 1$ in \cite{Ch0, Ch1},
\item $m=5,7$ in \cite{Ch0, Ch3}.
\end{itemize} 
In each case, the necessary existence condition (\ref{nec}) turns out to be sufficient.

\subsection{Contents} In this paper, we solve the currently smallest open case of the problem, namely the case $m=4$. Our solution is presented in Section~\ref{solution} and proved valid in Section~\ref{proof}. Here again, the necessary existence condition \eqref{nec} is found to be sufficient. The construction method, which consists in attempting to lift to $\Z{4}$ specific known solutions in $\Z{2}$, is explained in Section~\ref{method}. It will probably take some time before complete solutions emerge for more moduli. For this reason we propose, in a short concluding section, a hopefully more tractable version of the problem.

\section{A solution for $m=4$}\label{solution}

Here we solve Molluzzo's problem in the group $\Z{4}$. For $m=4$, it is easy to see that the necessary condition (\ref{nec}) amounts to the following: for all $n \in \N$, we have
$$
{n+1 \choose 2} \,\equiv\, 0 \bmod 4 \;\; \iff \;\; n \,\equiv\, 0 \textrm{ or } 7 \bmod 8.
$$

As in \cite{EH} for the case $m=2$, our solution involves the concept of \textit{strongly balanced} triangles. We first introduce a notation for initial segments of sequences.

\begin{notation}
Let $S=(x_i)_{i \ge 1}$ be a finite or infinite sequence, and let $l \ge 0$ be an integer not exceeding the length of $S$. We denote by $S[l]=(x_1,\ldots,x_l)$ the initial segment of length $l$ of $S$.
\end{notation}

\begin{defn}\label{strong}
Let $S$ be a finite sequence of length $n\ge 0$ in $\Z{4}$. The Steinhaus triangle $\nabla S$ is said to be \textit{strongly balanced} if, for every $0\le t \le n/8$, the Steinhaus triangle $\nabla S[n-8t]$ is balanced.
\end{defn}

Here is our main result in this paper.

\begin{thm}\label{thm1} There exists a balanced Steinhaus triangle of length $n$ in $\Z{4}$ if and only if ${n+1 \choose 2} \,\equiv\, 0 \bmod 4$. More precisely, consider the following infinite, eventually periodic sequences in $\Z{4}$:
$$
\begin{array}{l}
S_1 = 01220232(212113220030232311200232)^{\infty},\\
\medskip
S_2 = 21210130(200132022112002110220130)^{\infty},\\
T_1 = 0120021(212202102023032200322021)^{\infty},\\
T_2 = 1000212(312223301210312003103232)^{\infty},\\
T_3 = 1200210(220101222032222103000210)^{\infty},\\
T_4 = 2102203(232002102021230022302203)^{\infty}.\\
\end{array}
$$
Then, for all integers $i,j,k$ such that $1 \le i \le 2$, $1 \le j \le 4$ and $k \ge 0$, the Steinhaus triangles of the initial segments $S_i[8k]$ and $T_j[8k+7]$ are strongly balanced.
\end{thm}

\section{Proof of Theorem~\ref{thm1}}\label{proof}

We now prove the above theorem. Actually, only the statements concerning $S_1$ are treated in detail. The proof method for the other sequences $S_2$, $T_1$, $T_2$, $T_3$, $T_4$ is similar and only briefly commented.

\begin{notation}
Let $A,B$ be two \textit{blocks}, either triangles or lozenges. We denote by $A\ast B$ the unique parallelogram they would determine by Pascal's rule (C1) if they were adjacent in a large Steinhaus triangle (as may be freely assumed to be the case). See Figure~\ref{def A*B}.
\end{notation}

\begin{figure}
\begin{center}
\begin{tikzpicture}[scale=0.25]

\pgfmathparse{sqrt(3)}\let\x\pgfmathresult

\draw (0,10*\x) -- (10,0) -- (20,10*\x);
\draw (4,6*\x) -- (8,10*\x) -- (14,4*\x);

\node at (4,26*\x/3) {$A$};
\node at (14,8*\x) {$B$};
\node at (9,5*\x) {$A\ast B$};
\end{tikzpicture}
\caption{\label{def A*B}Defining $A\ast B$.}
\end{center}
\end{figure}
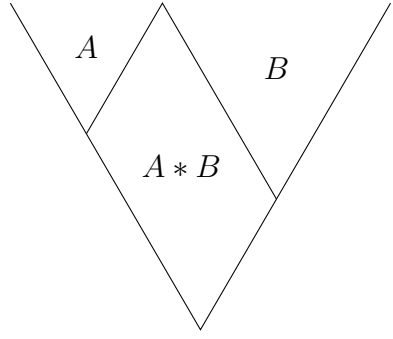

Note that $A \ast B$ only depends on the right lower side of $A$ and the left lower side of $B$, and is a lozenge if $A,B$ have the same side length. For example:
\begin{itemize}
\item
if $A =$
\begin{tikzpicture}[scale=0.25]
\pgfmathparse{sqrt(3)}\let\x\pgfmathresult
\draw (0,\x) node {$0$};
\draw (2,\x) node {$1$};
\draw (1,0) node {$1$};
\end{tikzpicture}
and $B=$
\begin{tikzpicture}[scale=0.25]
\pgfmathparse{sqrt(3)}\let\x\pgfmathresult
\draw (0,\x) node {$2$};
\draw (2,\x) node {$3$};
\draw (1,0) node {$1$};
\end{tikzpicture}
, then $A\ast B =$
\begin{tikzpicture}[scale=0.25]
\pgfmathparse{sqrt(3)}\let\x\pgfmathresult
\draw (1,2*\x) node {$3$};
\draw (0,\x) node {$0$};
\draw (2,\x) node {$0$};
\draw (1,0) node {$0$};
\end{tikzpicture}
since
\begin{tikzpicture}[scale=0.25]
\pgfmathparse{sqrt(3)}\let\x\pgfmathresult
\draw (0,3*\x) node {$0$};
\draw (2,3*\x) node {$1$};
\draw (4,3*\x) node {$2$};
\draw (6,3*\x) node {$3$};
\draw (1,2*\x) node {$1$};
\draw (3,2*\x) node {$\mathbf{3}$};
\draw (5,2*\x) node {$1$};
\draw (2,\x) node {$\mathbf{0}$};
\draw (4,\x) node {$\mathbf{0}$};
\draw (3,0) node {$\mathbf{0}$};
\end{tikzpicture};

\item
if $A=$
\begin{tikzpicture}[scale=0.25]
\pgfmathparse{sqrt(3)}\let\x\pgfmathresult
\draw (1,2*\x) node {$2$};
\draw (0,\x) node {$2$};
\draw (2,\x) node {$3$};
\draw (1,0) node {$1$};
\end{tikzpicture}
and $B=$
\begin{tikzpicture}[scale=0.25]
\pgfmathparse{sqrt(3)}\let\x\pgfmathresult
\draw (1,2*\x) node {$3$};
\draw (0,\x) node {$0$};
\draw (2,\x) node {$2$};
\draw (1,0) node {$2$};
\end{tikzpicture}
, then $A\ast B=$
\begin{tikzpicture}[scale=0.25]
\pgfmathparse{sqrt(3)}\let\x\pgfmathresult
\draw (1,2*\x) node {$3$};
\draw (0,\x) node {$0$};
\draw (2,\x) node {$1$};
\draw (1,0) node {$1$};
\end{tikzpicture}
since
\begin{tikzpicture}[scale=0.25]
\pgfmathparse{sqrt(3)}\let\x\pgfmathresult
\draw (1,4*\x) node {$2$};
\draw (0,3*\x) node {$2$};
\draw (2,3*\x) node {$3$};
\draw (1,2*\x) node {$1$};

\draw (5,4*\x) node {$3$};
\draw (4,3*\x) node {$0$};
\draw (6,3*\x) node {$2$};
\draw (5,2*\x) node {$2$};

\draw (3,2*\x) node {$\mathbf{3}$};
\draw (2,\x) node {$\mathbf{0}$};
\draw (4,\x) node {$\mathbf{1}$};
\draw (3,0) node {$\mathbf{1}$};

\end{tikzpicture}.
\end{itemize}

Consider now the four triangular blocks $A_0,A_1,A_2,A_3$ depicted in Figure~\ref{fig2}. Taking the $\ast$ product of selected pairs, we define
\begin{equation}\label{eq1}
B_i=A_i \ast A_{i+1} \;\; \textrm{for $i=0,1,2$ \;\; and \;\; } B_3=A_3 \ast A_1.
\end{equation}
Similarly, we set
\begin{equation}\label{eq2}
C_i=B_i \ast B_{i+1} \;\; \textrm{for $i=0,1,2$ \;\; and \;\; } C_3=B_3 \ast B_1.
\end{equation}
Finally, we also need to set
\begin{equation}\label{eq3}
D_0 = C_0 \ast C_1 \textrm{ \;\; and \;\; } E_0 = D_0 \ast C_3.
\end{equation}
The blocks $A_i,B_i,C_i,D_0,E_0$ $(i=1,2,3)$ are displayed in Figure~\ref{fig2}. In view of the following lemma, we shall refer to them as the \textit{building blocks} of $\nabla S_1[8k]$. 

\begin{lem}\label{lem1}
For every integer $k \ge 0$, the Steinhaus triangle $\nabla S_1[8k]$ has the structure depicted in Figure~\ref{fig1}, where $A_i, B_i, C_i$ and $D_0, E_0$ are the blocks defined above. 
\end{lem}

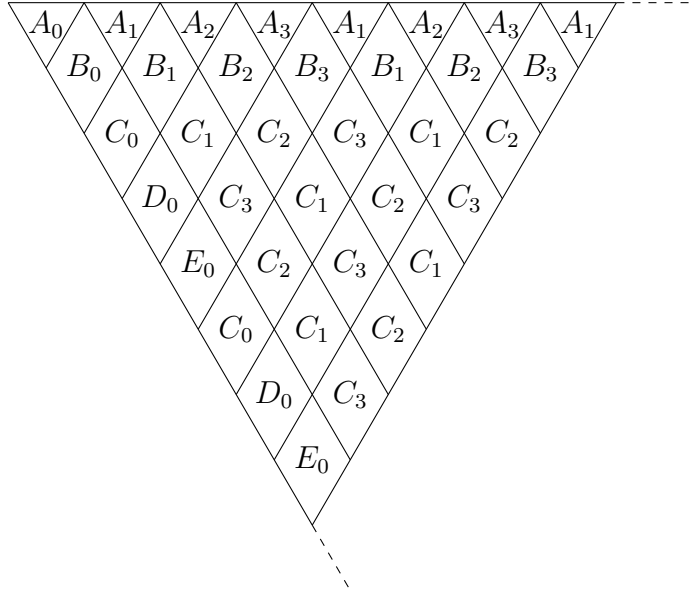
\begin{figure}[!h]
\begin{center}
\begin{tikzpicture}
\pgfmathparse{sqrt(3)/2}\let\x\pgfmathresult
\draw (0,9*\x) -- (8,9*\x);
\draw[dashed] (8,9*\x) -- (9,9*\x);
\draw[dashed] (4,\x) -- (4.5,0);
\foreach \i in {0,...,7}
{
	\pgfmathparse{\i /2}\let\j\pgfmathresult
	\pgfmathparse{\i +1}\let\k\pgfmathresult
	\draw (\i ,9*\x) -- (4+\j ,\k *\x);
}
\foreach \i in {1,...,8}
{
	\pgfmathparse{\i /2}\let\j\pgfmathresult
	\pgfmathparse{9-\i}\let\k\pgfmathresult
	\draw (\i ,9*\x) -- (\j ,\k *\x);
}

\pgfmathparse{26/3*\x}\let\y\pgfmathresult
\draw (0.5,\y) node {$A_0$};
\draw (1.5,\y) node {$A_1$};
\draw (2.5,\y) node {$A_2$};
\draw (3.5,\y) node {$A_3$};
\draw (4.5,\y) node {$A_1$};
\draw (5.5,\y) node {$A_2$};
\draw (6.5,\y) node {$A_3$};
\draw (7.5,\y) node {$A_1$};

\draw (1,8*\x) node {$B_0$};
\draw (2,8*\x) node {$B_1$};
\draw (3,8*\x) node {$B_2$};
\draw (4,8*\x) node {$B_3$};
\draw (5,8*\x) node {$B_1$};
\draw (6,8*\x) node {$B_2$};
\draw (7,8*\x) node {$B_3$};

\draw (1.5,7*\x) node {$C_0$};
\draw (2.5,7*\x) node {$C_1$};
\draw (3.5,7*\x) node {$C_2$};
\draw (4.5,7*\x) node {$C_3$};
\draw (5.5,7*\x) node {$C_1$};
\draw (6.5,7*\x) node {$C_2$};

\draw (2,6*\x) node {$D_0$};
\draw (3,6*\x) node {$C_3$};
\draw (4,6*\x) node {$C_1$};
\draw (5,6*\x) node {$C_2$};
\draw (6,6*\x) node {$C_3$};

\draw (2.5,5*\x) node {$E_0$};
\draw (3.5,5*\x) node {$C_2$};
\draw (4.5,5*\x) node {$C_3$};
\draw (5.5,5*\x) node {$C_1$};

\draw (3,4*\x) node {$C_0$};
\draw (4,4*\x) node {$C_1$};
\draw (5,4*\x) node {$C_2$};

\draw (3.5,3*\x) node {$D_0$};
\draw (4.5,3*\x) node {$C_3$};

\draw (4,2*\x) node {$E_0$};

\end{tikzpicture}
\end{center}
\caption{\label{fig1}Structure of $\nabla S_1[8k]$}
\end{figure}

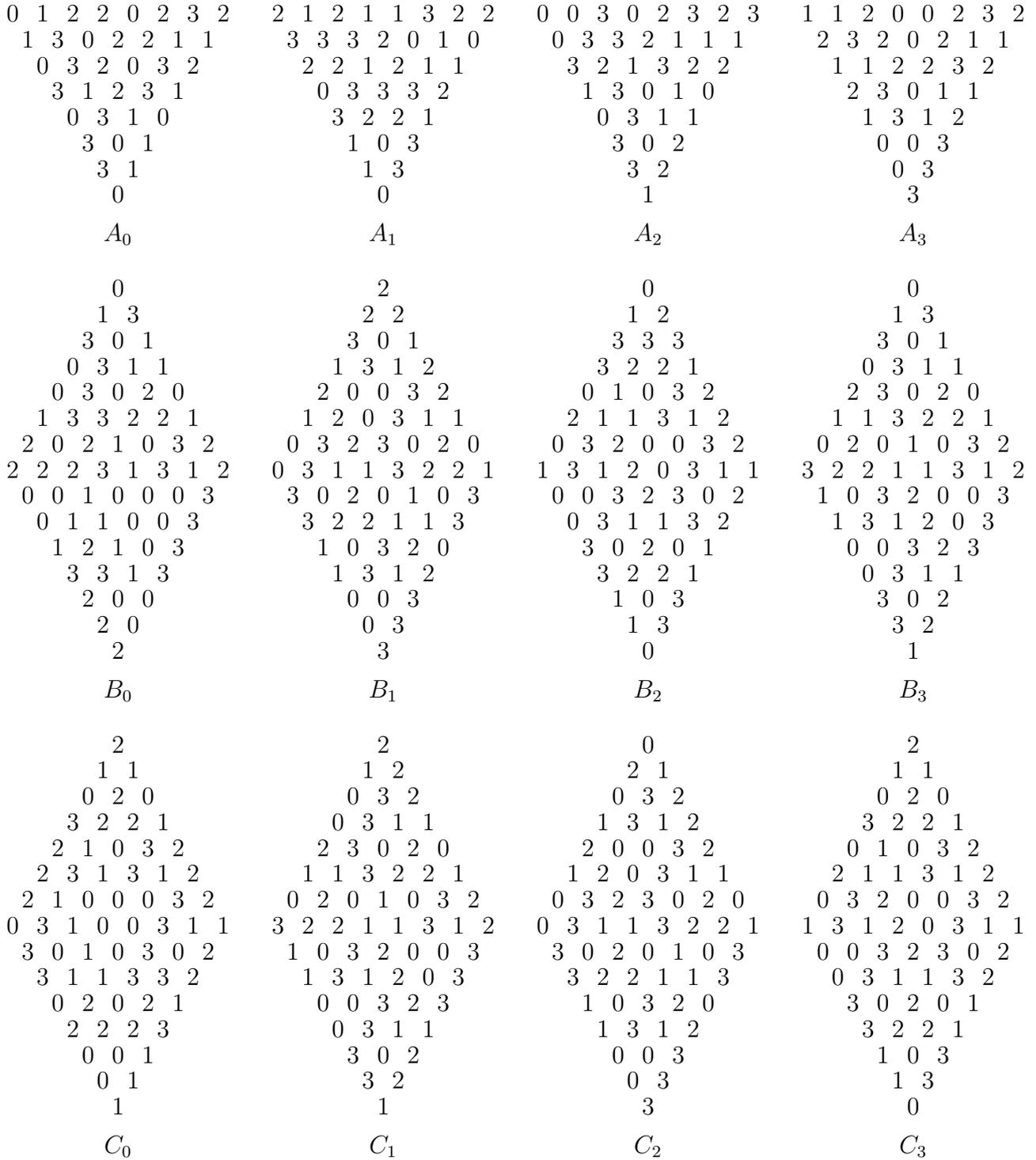
\begin{figure}
\hspace{-30pt}
\begin{tabular}{cccc}

\begin{tikzpicture}[scale=0.25]
\pgfmathparse{sqrt(3)}\let\y\pgfmathresult
\setcounter{x}{0}
\foreach \i in {0,1,2,2,0,2,3,2}
{
	\draw (\thex,7*\y) node {$ \i $};
	\addtocounter{x}{2}
}
\setcounter{x}{1}
\foreach \i in {1,3,0,2,2,1,1}
{
	\draw (\thex,6*\y) node {$ \i $};
	\addtocounter{x}{2}
}
\setcounter{x}{2}
\foreach \i in {0,3,2,0,3,2}
{
	\draw (\thex,5*\y) node {$ \i $};
	\addtocounter{x}{2}
}
\setcounter{x}{3}
\foreach \i in {3,1,2,3,1}
{
	\draw (\thex,4*\y) node {$ \i $};
	\addtocounter{x}{2}
}
\setcounter{x}{4}
\foreach \i in {0,3,1,0}
{
	\draw (\thex,3*\y) node {$ \i $};
	\addtocounter{x}{2}
}
\setcounter{x}{5}
\foreach \i in {3,0,1}
{
	\draw (\thex,2*\y) node {$ \i $};
	\addtocounter{x}{2}
}
\setcounter{x}{6}
\foreach \i in {3,1}
{
	\draw (\thex,\y) node {$ \i $};
	\addtocounter{x}{2}
}
\setcounter{x}{7}
\foreach \i in {0}
{
	\draw (\thex,0) node {$ \i $};
	\addtocounter{x}{2}
}
\end{tikzpicture}

&

\begin{tikzpicture}[scale=0.25]
\pgfmathparse{sqrt(3)}\let\y\pgfmathresult
\setcounter{x}{0}
\foreach \i in {2,1,2,1,1,3,2,2}
{
	\draw (\thex,7*\y) node {$ \i $};
	\addtocounter{x}{2}
}
\setcounter{x}{1}
\foreach \i in {3,3,3,2,0,1,0}
{
	\draw (\thex,6*\y) node {$ \i $};
	\addtocounter{x}{2}
}
\setcounter{x}{2}
\foreach \i in {2,2,1,2,1,1}
{
	\draw (\thex,5*\y) node {$ \i $};
	\addtocounter{x}{2}
}
\setcounter{x}{3}
\foreach \i in {0,3,3,3,2}
{
	\draw (\thex,4*\y) node {$ \i $};
	\addtocounter{x}{2}
}
\setcounter{x}{4}
\foreach \i in {3,2,2,1}
{
	\draw (\thex,3*\y) node {$ \i $};
	\addtocounter{x}{2}
}
\setcounter{x}{5}
\foreach \i in {1,0,3}
{
	\draw (\thex,2*\y) node {$ \i $};
	\addtocounter{x}{2}
}
\setcounter{x}{6}
\foreach \i in {1,3}
{
	\draw (\thex,\y) node {$ \i $};
	\addtocounter{x}{2}
}
\setcounter{x}{7}
\foreach \i in {0}
{
	\draw (\thex,0) node {$ \i $};
	\addtocounter{x}{2}
}
\end{tikzpicture}

&

\begin{tikzpicture}[scale=0.25]
\pgfmathparse{sqrt(3)}\let\y\pgfmathresult
\setcounter{x}{0}
\foreach \i in {0,0,3,0,2,3,2,3}
{
	\draw (\thex,7*\y) node {$ \i $};
	\addtocounter{x}{2}
}
\setcounter{x}{1}
\foreach \i in {0,3,3,2,1,1,1}
{
	\draw (\thex,6*\y) node {$ \i $};
	\addtocounter{x}{2}
}
\setcounter{x}{2}
\foreach \i in {3,2,1,3,2,2}
{
	\draw (\thex,5*\y) node {$ \i $};
	\addtocounter{x}{2}
}
\setcounter{x}{3}
\foreach \i in {1,3,0,1,0}
{
	\draw (\thex,4*\y) node {$ \i $};
	\addtocounter{x}{2}
}
\setcounter{x}{4}
\foreach \i in {0,3,1,1}
{
	\draw (\thex,3*\y) node {$ \i $};
	\addtocounter{x}{2}
}
\setcounter{x}{5}
\foreach \i in {3,0,2}
{
	\draw (\thex,2*\y) node {$ \i $};
	\addtocounter{x}{2}
}
\setcounter{x}{6}
\foreach \i in {3,2}
{
	\draw (\thex,\y) node {$ \i $};
	\addtocounter{x}{2}
}
\setcounter{x}{7}
\foreach \i in {1}
{
	\draw (\thex,0) node {$ \i $};
	\addtocounter{x}{2}
}
\end{tikzpicture}

&

\begin{tikzpicture}[scale=0.25]
\pgfmathparse{sqrt(3)}\let\y\pgfmathresult
\setcounter{x}{0}
\foreach \i in {1,1,2,0,0,2,3,2}
{
	\draw (\thex,7*\y) node {$ \i $};
	\addtocounter{x}{2}
}
\setcounter{x}{1}
\foreach \i in {2,3,2,0,2,1,1}
{
	\draw (\thex,6*\y) node {$ \i $};
	\addtocounter{x}{2}
}
\setcounter{x}{2}
\foreach \i in {1,1,2,2,3,2}
{
	\draw (\thex,5*\y) node {$ \i $};
	\addtocounter{x}{2}
}
\setcounter{x}{3}
\foreach \i in {2,3,0,1,1}
{
	\draw (\thex,4*\y) node {$ \i $};
	\addtocounter{x}{2}
}
\setcounter{x}{4}
\foreach \i in {1,3,1,2}
{
	\draw (\thex,3*\y) node {$ \i $};
	\addtocounter{x}{2}
}
\setcounter{x}{5}
\foreach \i in {0,0,3}
{
	\draw (\thex,2*\y) node {$ \i $};
	\addtocounter{x}{2}
}
\setcounter{x}{6}
\foreach \i in {0,3}
{
	\draw (\thex,\y) node {$ \i $};
	\addtocounter{x}{2}
}
\setcounter{x}{7}
\foreach \i in {3}
{
	\draw (\thex,0) node {$ \i $};
	\addtocounter{x}{2}
}
\end{tikzpicture}
\\

$A_0$ & $A_1$ & $A_2$ & $A_3$ \\[2ex]

\begin{tikzpicture}[scale=0.25]
\pgfmathparse{sqrt(3)}\let\y\pgfmathresult
\setcounter{x}{7}
\foreach \i in {0}
{
	\draw (\thex,14*\y) node {$ \i $};
	\addtocounter{x}{2}
}
\setcounter{x}{6}
\foreach \i in {1,3}
{
	\draw (\thex,13*\y) node {$ \i $};
	\addtocounter{x}{2}
}
\setcounter{x}{5}
\foreach \i in {3,0,1}
{
	\draw (\thex,12*\y) node {$ \i $};
	\addtocounter{x}{2}
}
\setcounter{x}{4}
\foreach \i in {0,3,1,1}
{
	\draw (\thex,11*\y) node {$ \i $};
	\addtocounter{x}{2}
}
\setcounter{x}{3}
\foreach \i in {0,3,0,2,0}
{
	\draw (\thex,10*\y) node {$ \i $};
	\addtocounter{x}{2}
}
\setcounter{x}{2}
\foreach \i in {1,3,3,2,2,1}
{
	\draw (\thex,9*\y) node {$ \i $};
	\addtocounter{x}{2}
}
\setcounter{x}{1}
\foreach \i in {2,0,2,1,0,3,2}
{
	\draw (\thex,8*\y) node {$ \i $};
	\addtocounter{x}{2}
}
\setcounter{x}{0}
\foreach \i in {2,2,2,3,1,3,1,2}
{
	\draw (\thex,7*\y) node {$ \i $};
	\addtocounter{x}{2}
}
\setcounter{x}{1}
\foreach \i in {0,0,1,0,0,0,3}
{
	\draw (\thex,6*\y) node {$ \i $};
	\addtocounter{x}{2}
}
\setcounter{x}{2}
\foreach \i in {0,1,1,0,0,3}
{
	\draw (\thex,5*\y) node {$ \i $};
	\addtocounter{x}{2}
}
\setcounter{x}{3}
\foreach \i in {1,2,1,0,3}
{
	\draw (\thex,4*\y) node {$ \i $};
	\addtocounter{x}{2}
}
\setcounter{x}{4}
\foreach \i in {3,3,1,3}
{
	\draw (\thex,3*\y) node {$ \i $};
	\addtocounter{x}{2}
}
\setcounter{x}{5}
\foreach \i in {2,0,0}
{
	\draw (\thex,2*\y) node {$ \i $};
	\addtocounter{x}{2}
}
\setcounter{x}{6}
\foreach \i in {2,0}
{
	\draw (\thex,\y) node {$ \i $};
	\addtocounter{x}{2}
}
\setcounter{x}{7}
\foreach \i in {2}
{
	\draw (\thex,0) node {$ \i $};
	\addtocounter{x}{2}
}
\end{tikzpicture}

&

\begin{tikzpicture}[scale=0.25]
\pgfmathparse{sqrt(3)}\let\y\pgfmathresult
\setcounter{x}{7}
\foreach \i in {2}
{
	\draw (\thex,14*\y) node {$ \i $};
	\addtocounter{x}{2}
}
\setcounter{x}{6}
\foreach \i in {2,2}
{
	\draw (\thex,13*\y) node {$ \i $};
	\addtocounter{x}{2}
}
\setcounter{x}{5}
\foreach \i in {3,0,1}
{
	\draw (\thex,12*\y) node {$ \i $};
	\addtocounter{x}{2}
}
\setcounter{x}{4}
\foreach \i in {1,3,1,2}
{
	\draw (\thex,11*\y) node {$ \i $};
	\addtocounter{x}{2}
}
\setcounter{x}{3}
\foreach \i in {2,0,0,3,2}
{
	\draw (\thex,10*\y) node {$ \i $};
	\addtocounter{x}{2}
}
\setcounter{x}{2}
\foreach \i in {1,2,0,3,1,1}
{
	\draw (\thex,9*\y) node {$ \i $};
	\addtocounter{x}{2}
}
\setcounter{x}{1}
\foreach \i in {0,3,2,3,0,2,0}
{
	\draw (\thex,8*\y) node {$ \i $};
	\addtocounter{x}{2}
}
\setcounter{x}{0}
\foreach \i in {0,3,1,1,3,2,2,1}
{
	\draw (\thex,7*\y) node {$ \i $};
	\addtocounter{x}{2}
}
\setcounter{x}{1}
\foreach \i in {3,0,2,0,1,0,3}
{
	\draw (\thex,6*\y) node {$ \i $};
	\addtocounter{x}{2}
}
\setcounter{x}{2}
\foreach \i in {3,2,2,1,1,3}
{
	\draw (\thex,5*\y) node {$ \i $};
	\addtocounter{x}{2}
}
\setcounter{x}{3}
\foreach \i in {1,0,3,2,0}
{
	\draw (\thex,4*\y) node {$ \i $};
	\addtocounter{x}{2}
}
\setcounter{x}{4}
\foreach \i in {1,3,1,2}
{
	\draw (\thex,3*\y) node {$ \i $};
	\addtocounter{x}{2}
}
\setcounter{x}{5}
\foreach \i in {0,0,3}
{
	\draw (\thex,2*\y) node {$ \i $};
	\addtocounter{x}{2}
}
\setcounter{x}{6}
\foreach \i in {0,3}
{
	\draw (\thex,\y) node {$ \i $};
	\addtocounter{x}{2}
}
\setcounter{x}{7}
\foreach \i in {3}
{
	\draw (\thex,0) node {$ \i $};
	\addtocounter{x}{2}
}
\end{tikzpicture}

&

\begin{tikzpicture}[scale=0.25]
\pgfmathparse{sqrt(3)}\let\y\pgfmathresult
\setcounter{x}{7}
\foreach \i in {0}
{
	\draw (\thex,14*\y) node {$ \i $};
	\addtocounter{x}{2}
}
\setcounter{x}{6}
\foreach \i in {1,2}
{
	\draw (\thex,13*\y) node {$ \i $};
	\addtocounter{x}{2}
}
\setcounter{x}{5}
\foreach \i in {3,3,3}
{
	\draw (\thex,12*\y) node {$ \i $};
	\addtocounter{x}{2}
}
\setcounter{x}{4}
\foreach \i in {3,2,2,1}
{
	\draw (\thex,11*\y) node {$ \i $};
	\addtocounter{x}{2}
}
\setcounter{x}{3}
\foreach \i in {0,1,0,3,2}
{
	\draw (\thex,10*\y) node {$ \i $};
	\addtocounter{x}{2}
}
\setcounter{x}{2}
\foreach \i in {2,1,1,3,1,2}
{
	\draw (\thex,9*\y) node {$ \i $};
	\addtocounter{x}{2}
}
\setcounter{x}{1}
\foreach \i in {0,3,2,0,0,3,2}
{
	\draw (\thex,8*\y) node {$ \i $};
	\addtocounter{x}{2}
}
\setcounter{x}{0}
\foreach \i in {1,3,1,2,0,3,1,1}
{
	\draw (\thex,7*\y) node {$ \i $};
	\addtocounter{x}{2}
}
\setcounter{x}{1}
\foreach \i in {0,0,3,2,3,0,2}
{
	\draw (\thex,6*\y) node {$ \i $};
	\addtocounter{x}{2}
}
\setcounter{x}{2}
\foreach \i in {0,3,1,1,3,2}
{
	\draw (\thex,5*\y) node {$ \i $};
	\addtocounter{x}{2}
}
\setcounter{x}{3}
\foreach \i in {3,0,2,0,1}
{
	\draw (\thex,4*\y) node {$ \i $};
	\addtocounter{x}{2}
}
\setcounter{x}{4}
\foreach \i in {3,2,2,1}
{
	\draw (\thex,3*\y) node {$ \i $};
	\addtocounter{x}{2}
}
\setcounter{x}{5}
\foreach \i in {1,0,3}
{
	\draw (\thex,2*\y) node {$ \i $};
	\addtocounter{x}{2}
}
\setcounter{x}{6}
\foreach \i in {1,3}
{
	\draw (\thex,\y) node {$ \i $};
	\addtocounter{x}{2}
}
\setcounter{x}{7}
\foreach \i in {0}
{
	\draw (\thex,0) node {$ \i $};
	\addtocounter{x}{2}
}
\end{tikzpicture}

&

\begin{tikzpicture}[scale=0.25]
\pgfmathparse{sqrt(3)}\let\y\pgfmathresult
\setcounter{x}{7}
\foreach \i in {0}
{
	\draw (\thex , 14*\y) node {$ \i $};
	\addtocounter{x}{2}
}
\setcounter{x}{6}
\foreach \i in {1,3}
{
	\draw (\thex , 13*\y) node {$ \i $};
	\addtocounter{x}{2}
}
\setcounter{x}{5}
\foreach \i in {3,0,1}
{
	\draw (\thex , 12*\y) node {$ \i $};
	\addtocounter{x}{2}
}
\setcounter{x}{4}
\foreach \i in {0,3,1,1}
{
	\draw (\thex , 11*\y) node {$ \i $};
	\addtocounter{x}{2}
}
\setcounter{x}{3}
\foreach \i in {2,3,0,2,0}
{
	\draw (\thex , 10*\y) node {$ \i $};
	\addtocounter{x}{2}
}
\setcounter{x}{2}
\foreach \i in {1,1,3,2,2,1}
{
	\draw (\thex , 9*\y) node {$ \i $};
	\addtocounter{x}{2}
}
\setcounter{x}{1}
\foreach \i in {0,2,0,1,0,3,2}
{
	\draw (\thex , 8*\y) node {$ \i $};
	\addtocounter{x}{2}
}
\setcounter{x}{0}
\foreach \i in {3,2,2,1,1,3,1,2}
{
	\draw (\thex , 7*\y) node {$ \i $};
	\addtocounter{x}{2}
}
\setcounter{x}{1}
\foreach \i in {1,0,3,2,0,0,3}
{
	\draw (\thex , 6*\y) node {$ \i $};
	\addtocounter{x}{2}
}
\setcounter{x}{2}
\foreach \i in {1,3,1,2,0,3}
{
	\draw (\thex , 5*\y) node {$ \i $};
	\addtocounter{x}{2}
}
\setcounter{x}{3}
\foreach \i in {0,0,3,2,3}
{
	\draw (\thex , 4*\y) node {$ \i $};
	\addtocounter{x}{2}
}
\setcounter{x}{4}
\foreach \i in {0,3,1,1}
{
	\draw (\thex , 3*\y) node {$ \i $};
	\addtocounter{x}{2}
}
\setcounter{x}{5}
\foreach \i in {3,0,2}
{
	\draw (\thex , 2*\y) node {$ \i $};
	\addtocounter{x}{2}
}
\setcounter{x}{6}
\foreach \i in {3,2}
{
	\draw (\thex , \y) node {$ \i $};
	\addtocounter{x}{2}
}
\setcounter{x}{7}
\foreach \i in {1}
{
	\draw (\thex , 0) node {$ \i $};
	\addtocounter{x}{2}
}
\end{tikzpicture}\\

$B_0$ & $B_1$ & $B_2$ & $B_3$ \\[2ex]

\begin{tikzpicture}[scale=0.25]
\pgfmathparse{sqrt(3)}\let\y\pgfmathresult
\setcounter{x}{7}
\foreach \i in {2}
{
	\draw (\thex,14*\y) node {$ \i $};
	\addtocounter{x}{2}
}
\setcounter{x}{6}
\foreach \i in {1,1}
{
	\draw (\thex,13*\y) node {$ \i $};
	\addtocounter{x}{2}
}
\setcounter{x}{5}
\foreach \i in {0,2,0}
{
	\draw (\thex,12*\y) node {$ \i $};
	\addtocounter{x}{2}
}
\setcounter{x}{4}
\foreach \i in {3,2,2,1}
{
	\draw (\thex,11*\y) node {$ \i $};
	\addtocounter{x}{2}
}
\setcounter{x}{3}
\foreach \i in {2,1,0,3,2}
{
	\draw (\thex,10*\y) node {$ \i $};
	\addtocounter{x}{2}
}
\setcounter{x}{2}
\foreach \i in {2,3,1,3,1,2}
{
	\draw (\thex,9*\y) node {$ \i $};
	\addtocounter{x}{2}
}
\setcounter{x}{1}
\foreach \i in {2,1,0,0,0,3,2}
{
	\draw (\thex,8*\y) node {$ \i $};
	\addtocounter{x}{2}
}
\setcounter{x}{0}
\foreach \i in {0,3,1,0,0,3,1,1}
{
	\draw (\thex,7*\y) node {$ \i $};
	\addtocounter{x}{2}
}
\setcounter{x}{1}
\foreach \i in {3,0,1,0,3,0,2}
{
	\draw (\thex,6*\y) node {$ \i $};
	\addtocounter{x}{2}
}
\setcounter{x}{2}
\foreach \i in {3,1,1,3,3,2}
{
	\draw (\thex,5*\y) node {$ \i $};
	\addtocounter{x}{2}
}
\setcounter{x}{3}
\foreach \i in {0,2,0,2,1}
{
	\draw (\thex,4*\y) node {$ \i $};
	\addtocounter{x}{2}
}
\setcounter{x}{4}
\foreach \i in {2,2,2,3}
{
	\draw (\thex,3*\y) node {$ \i $};
	\addtocounter{x}{2}
}
\setcounter{x}{5}
\foreach \i in {0,0,1}
{
	\draw (\thex,2*\y) node {$ \i $};
	\addtocounter{x}{2}
}
\setcounter{x}{6}
\foreach \i in {0,1}
{
	\draw (\thex,\y) node {$ \i $};
	\addtocounter{x}{2}
}
\setcounter{x}{7}
\foreach \i in {1}
{
	\draw (\thex,0) node {$ \i $};
	\addtocounter{x}{2}
}
\end{tikzpicture}

& 

\begin{tikzpicture}[scale=0.25]
\pgfmathparse{sqrt(3)}\let\y\pgfmathresult
\setcounter{x}{7}
\foreach \i in {2}
{
	\draw (\thex,14*\y) node {$ \i $};
	\addtocounter{x}{2}
}
\setcounter{x}{6}
\foreach \i in {1,2}
{
	\draw (\thex,13*\y) node {$ \i $};
	\addtocounter{x}{2}
}
\setcounter{x}{5}
\foreach \i in {0,3,2}
{
	\draw (\thex,12*\y) node {$ \i $};
	\addtocounter{x}{2}
}
\setcounter{x}{4}
\foreach \i in {0,3,1,1}
{
	\draw (\thex,11*\y) node {$ \i $};
	\addtocounter{x}{2}
}
\setcounter{x}{3}
\foreach \i in {2,3,0,2,0}
{
	\draw (\thex,10*\y) node {$ \i $};
	\addtocounter{x}{2}
}
\setcounter{x}{2}
\foreach \i in {1,1,3,2,2,1}
{
	\draw (\thex,9*\y) node {$ \i $};
	\addtocounter{x}{2}
}
\setcounter{x}{1}
\foreach \i in {0,2,0,1,0,3,2}
{
	\draw (\thex,8*\y) node {$ \i $};
	\addtocounter{x}{2}
}
\setcounter{x}{0}
\foreach \i in {3,2,2,1,1,3,1,2}
{
	\draw (\thex,7*\y) node {$ \i $};
	\addtocounter{x}{2}
}
\setcounter{x}{1}
\foreach \i in {1,0,3,2,0,0,3}
{
	\draw (\thex,6*\y) node {$ \i $};
	\addtocounter{x}{2}
}
\setcounter{x}{2}
\foreach \i in {1,3,1,2,0,3}
{
	\draw (\thex,5*\y) node {$ \i $};
	\addtocounter{x}{2}
}
\setcounter{x}{3}
\foreach \i in {0,0,3,2,3}
{
	\draw (\thex,4*\y) node {$ \i $};
	\addtocounter{x}{2}
}
\setcounter{x}{4}
\foreach \i in {0,3,1,1}
{
	\draw (\thex,3*\y) node {$ \i $};
	\addtocounter{x}{2}
}
\setcounter{x}{5}
\foreach \i in {3,0,2}
{
	\draw (\thex ,2*\y) node {$ \i $};
	\addtocounter{x}{2}
}
\setcounter{x}{6}
\foreach \i in {3,2}
{
	\draw (\thex,\y) node {$ \i $};
	\addtocounter{x}{2}
}
\setcounter{x}{7}
\foreach \i in {1}
{
	\draw (\thex,0) node {$ \i $};
	\addtocounter{x}{2}
}
\end{tikzpicture}

& 

\begin{tikzpicture}[scale=0.25]
\pgfmathparse{sqrt(3)}\let\y\pgfmathresult
\setcounter{x}{7}
\foreach \i in {0}
{
	\draw (\thex,14*\y) node {$ \i $};
	\addtocounter{x}{2}
}
\setcounter{x}{6}
\foreach \i in {2,1}
{
	\draw (\thex,13*\y) node {$ \i $};
	\addtocounter{x}{2}
}
\setcounter{x}{5}
\foreach \i in {0,3,2}
{
	\draw (\thex, 12*\y) node {$ \i $};
	\addtocounter{x}{2}
}
\setcounter{x}{4}
\foreach \i in {1,3,1,2}
{
	\draw (\thex,11*\y) node {$ \i $};
	\addtocounter{x}{2}
}
\setcounter{x}{3}
\foreach \i in {2,0,0,3,2}
{
	\draw (\thex,10*\y) node {$ \i $};
	\addtocounter{x}{2}
}
\setcounter{x}{2}
\foreach \i in {1,2,0,3,1,1}
{
	\draw (\thex,9*\y) node {$ \i $};
	\addtocounter{x}{2}
}
\setcounter{x}{1}
\foreach \i in {0,3,2,3,0,2,0}
{
	\draw (\thex,8*\y) node {$ \i $};
	\addtocounter{x}{2}
}
\setcounter{x}{0}
\foreach \i in {0,3,1,1,3,2,2,1}
{
	\draw (\thex,7*\y) node {$ \i $};
	\addtocounter{x}{2}
}
\setcounter{x}{1}
\foreach \i in {3,0,2,0,1,0,3}
{
	\draw (\thex,6*\y) node {$ \i $};
	\addtocounter{x}{2}
}
\setcounter{x}{2}
\foreach \i in {3,2,2,1,1,3}
{
	\draw (\thex,5*\y) node {$ \i $};
	\addtocounter{x}{2}
}
\setcounter{x}{3}
\foreach \i in {1,0,3,2,0}
{
	\draw (\thex,4*\y) node {$ \i $};
	\addtocounter{x}{2}
}
\setcounter{x}{4}
\foreach \i in {1,3,1,2}
{
	\draw (\thex ,3*\y) node {$ \i $};
	\addtocounter{x}{2}
}
\setcounter{x}{5}
\foreach \i in {0,0,3}
{
	\draw (\thex,2*\y) node {$ \i $};
	\addtocounter{x}{2}
}
\setcounter{x}{6}
\foreach \i in {0,3}
{
	\draw (\thex,\y) node {$ \i $};
	\addtocounter{x}{2}
}
\setcounter{x}{7}
\foreach \i in {3}
{
	\draw (\thex,0) node {$ \i $};
	\addtocounter{x}{2}
}
\end{tikzpicture}

& 

\begin{tikzpicture}[scale=0.25]
\pgfmathparse{sqrt(3)}\let\y\pgfmathresult
\setcounter{x}{7}
\foreach \i in {2}
{
	\draw (\thex,14*\y) node {$ \i $};
	\addtocounter{x}{2}
}
\setcounter{x}{6}
\foreach \i in {1,1}
{
	\draw (\thex,13*\y) node {$ \i $};
	\addtocounter{x}{2}
}
\setcounter{x}{5}
\foreach \i in {0,2,0}
{
	\draw (\thex,12*\y) node {$ \i $};
	\addtocounter{x}{2}
}
\setcounter{x}{4}
\foreach \i in {3,2,2,1}
{
	\draw (\thex,11*\y) node {$ \i $};
	\addtocounter{x}{2}
}
\setcounter{x}{3}
\foreach \i in {0,1,0,3,2}
{
	\draw (\thex,10*\y) node {$ \i $};
	\addtocounter{x}{2}
}
\setcounter{x}{2}
\foreach \i in {2,1,1,3,1,2}
{
	\draw (\thex,9*\y) node {$ \i $};
	\addtocounter{x}{2}
}
\setcounter{x}{1}
\foreach \i in {0,3,2,0,0,3,2}
{
	\draw (\thex,8*\y) node {$ \i $};
	\addtocounter{x}{2}
}
\setcounter{x}{0}
\foreach \i in {1,3,1,2,0,3,1,1}
{
	\draw (\thex,7*\y) node {$ \i $};
	\addtocounter{x}{2}
}
\setcounter{x}{1}
\foreach \i in {0,0,3,2,3,0,2}
{
	\draw (\thex,6*\y) node {$ \i $};
	\addtocounter{x}{2}
}
\setcounter{x}{2}
\foreach \i in {0,3,1,1,3,2}
{
	\draw (\thex,5*\y) node {$ \i $};
	\addtocounter{x}{2}
}
\setcounter{x}{3}
\foreach \i in {3,0,2,0,1}
{
	\draw (\thex,4*\y) node {$ \i $};
	\addtocounter{x}{2}
}
\setcounter{x}{4}
\foreach \i in {3,2,2,1}
{
	\draw (\thex,3*\y) node {$ \i $};
	\addtocounter{x}{2}
}
\setcounter{x}{5}
\foreach \i in {1,0,3}
{
	\draw (\thex,2*\y) node {$ \i $};
	\addtocounter{x}{2}
}
\setcounter{x}{6}
\foreach \i in {1,3}
{
	\draw (\thex,\y) node {$ \i $};
	\addtocounter{x}{2}
}
\setcounter{x}{7}
\foreach \i in {0}
{
	\draw (\thex,0) node {$ \i $};
	\addtocounter{x}{2}
}
\end{tikzpicture}
\\

$C_0$ & $C_1$ & $C_2$ & $C_3$ \\[2ex]

\end{tabular}

\caption{\label{fig2}The building blocks of $\nabla S_1[8k]$}
\end{figure}

\begin{figure}

\begin{center}
\begin{tabular}{cccc}

\begin{tikzpicture}[scale=0.25]
\pgfmathparse{sqrt(3)}\let\y\pgfmathresult
\setcounter{x}{7}
\foreach \i in {0}
{
	\draw (\thex,14*\y) node {$ \i $};
	\addtocounter{x}{2}
}
\setcounter{x}{6}
\foreach \i in {2,1}
{
	\draw (\thex,13*\y) node {$ \i $};
	\addtocounter{x}{2}
}
\setcounter{x}{5}
\foreach \i in {0,3,2}
{
	\draw (\thex,12*\y) node {$ \i $};
	\addtocounter{x}{2}
}
\setcounter{x}{4}
\foreach \i in {1,3,1,2}
{
	\draw (\thex,11*\y) node {$ \i $};
	\addtocounter{x}{2}
}
\setcounter{x}{3}
\foreach \i in {0,0,0,3,2}
{
	\draw (\thex,10*\y) node {$ \i $};
	\addtocounter{x}{2}
}
\setcounter{x}{2}
\foreach \i in {1,0,0,3,1,1}
{
	\draw (\thex,9*\y) node {$ \i $};
	\addtocounter{x}{2}
}
\setcounter{x}{1}
\foreach \i in {2,1,0,3,0,2,0}
{
	\draw (\thex,8*\y) node {$ \i $};
	\addtocounter{x}{2}
}
\setcounter{x}{0}
\foreach \i in {3,3,1,3,3,2,2,1}
{
	\draw (\thex,7*\y) node {$ \i $};
	\addtocounter{x}{2}
}
\setcounter{x}{1}
\foreach \i in {2,0,0,2,1,0,3}
{
	\draw (\thex,6*\y) node {$ \i $};
	\addtocounter{x}{2}
}
\setcounter{x}{2}
\foreach \i in {2,0,2,3,1,3}
{
	\draw (\thex,5*\y) node {$ \i $};
	\addtocounter{x}{2}
}
\setcounter{x}{3}
\foreach \i in {2,2,1,0,0}
{
	\draw (\thex,4*\y) node {$ \i $};
	\addtocounter{x}{2}
}
\setcounter{x}{4}
\foreach \i in {0,3,1,0}
{
	\draw (\thex,3*\y) node {$ \i $};
	\addtocounter{x}{2}
}
\setcounter{x}{5}
\foreach \i in {3,0,1}
{
	\draw (\thex,2*\y) node {$ \i $};
	\addtocounter{x}{2}
}
\setcounter{x}{6}
\foreach \i in {3,1}
{
	\draw (\thex,\y) node {$ \i $};
	\addtocounter{x}{2}
}
\setcounter{x}{7}
\foreach \i in {0}
{
	\draw (\thex,0) node {$ \i $};
	\addtocounter{x}{2}
}
\end{tikzpicture}

& 

\begin{tikzpicture}[scale=0.25]
\pgfmathparse{sqrt(3)}\let\y\pgfmathresult
\setcounter{x}{7}
\foreach \i in {2}
{
	\draw (\thex,14*\y) node {$ \i $};
	\addtocounter{x}{2}
}
\setcounter{x}{6}
\foreach \i in {1,2}
{
	\draw (\thex,13*\y) node {$ \i $};
	\addtocounter{x}{2}
}
\setcounter{x}{5}
\foreach \i in {0,3,2}
{
	\draw (\thex,12*\y) node {$ \i $};
	\addtocounter{x}{2}
}
\setcounter{x}{4}
\foreach \i in {0,3,1,1}
{
	\draw (\thex,11*\y) node {$ \i $};
	\addtocounter{x}{2}
}
\setcounter{x}{3}
\foreach \i in {0,3,0,2,0}
{
	\draw (\thex,10*\y) node {$ \i $};
	\addtocounter{x}{2}
}
\setcounter{x}{2}
\foreach \i in {1,3,3,2,2,1}
{
	\draw (\thex,9*\y) node {$ \i $};
	\addtocounter{x}{2}
}
\setcounter{x}{1}
\foreach \i in {2,0,2,1,0,3,2}
{
	\draw (\thex,8*\y) node {$ \i $};
	\addtocounter{x}{2}
}
\setcounter{x}{0}
\foreach \i in {2,2,2,3,1,3,1,2}
{
	\draw (\thex,7*\y) node {$ \i $};
	\addtocounter{x}{2}
}
\setcounter{x}{1}
\foreach \i in {0,0,1,0,0,0,3}
{
	\draw (\thex,6*\y) node {$ \i $};
	\addtocounter{x}{2}
}
\setcounter{x}{2}
\foreach \i in {0,1,1,0,0,3}
{
	\draw (\thex,5*\y) node {$ \i $};
	\addtocounter{x}{2}
}
\setcounter{x}{3}
\foreach \i in {1,2,1,0,3}
{
	\draw (\thex,4*\y) node {$ \i $};
	\addtocounter{x}{2}
}
\setcounter{x}{4}
\foreach \i in {3,3,1,3}
{
	\draw (\thex,3*\y) node {$ \i $};
	\addtocounter{x}{2}
}
\setcounter{x}{5}
\foreach \i in {2,0,0}
{
	\draw (\thex,2*\y) node {$ \i $};
	\addtocounter{x}{2}
}
\setcounter{x}{6}
\foreach \i in {2,0}
{
	\draw (\thex,\y) node {$ \i $};
	\addtocounter{x}{2}
}
\setcounter{x}{7}
\foreach \i in {2}
{
	\draw (\thex,0) node {$ \i $};
	\addtocounter{x}{2}
}
\end{tikzpicture}

& & \\

$D_0$ & $E_0$ & & \\

\end{tabular}

\vspace{1cm}
Figure~\ref{fig2}: The building blocks of $\nabla S_1[8k]$ (continued)
\end{center}
\end{figure}

\begin{proof}
Recall that $S_1 = 01220232(212113220030232311200232)^{\infty}$. Thus, $S_1$ is made of an initial block $I=01220232$ of length 8, and a period $P_1P_2P_3$ of length $24=3 \times 8$, where 
$$
P_1=21211322,\;\; P_2=00302323,\;\; P_3=11200232.
$$ 
Observe that $A_0 = \nabla I$ and $A_i=\nabla P_i$ for $i=1,2,3$. This accounts for the top structure of $\nabla S_1[8k]$. Now, by definition we  have
$$
\begin{array}{cccc}
B_0 = A_0\ast A_1, & B_1 = A_1\ast A_2, & B_2 = A_2 \ast A_3, & B_3 = A_3 \ast A_1, \\
C_0 = B_0\ast B_1, & C_1 = B_1\ast B_2, & C_2 = B_2 \ast B_3, & C_3 = B_3 \ast B_1, \\
D_0 = C_0 \ast C_1, & E_0 = D_0 \ast C_3. & & \\
\end{array}
$$
It remains to show that $C_1\ast C_2 = C_3$, $C_2\ast C_3 = C_1$, $C_3\ast C_1 = C_2$ and $E_0\ast C_2 = C_0$. To do this, recall that the $\ast$ product of two blocks only depends on their lower sides, and observe on Figure~\ref{fig2} that 
\begin{enumerate}
\item[$-$] the lower sides of $C_1$ coincide with those of $B_3$;
\item[$-$] the lower sides of $C_2$ coincide with those of $B_1$;
\item[$-$] the lower sides of $C_3$ coincide with those of $B_2$;
\item[$-$] the lower sides of $E_0$ coincide with those of $B_0$.
\end{enumerate}
It follows that
\begin{equation}\label{eq4}
\begin{array}{ccc}
C_1\ast C_2 = B_3\ast B_1 = C_3, & & C_2\ast C_3 = B_1\ast B_2 = C_1, \\
C_3\ast C_1 = B_2\ast B_3 = C_2, & & E_0\ast C_2 = B_0\ast B_1 = C_0, \\
\end{array}
\end{equation}
as desired. This completes the proof of the lemma.
\end{proof}

We are now in a position to prove our main result.

\

\noindent\textit{Proof of Theorem~\ref{thm1} for $S_1$.} We shall prove, by induction on $k$, that the Steinhaus triangle $\nabla S_1[8k]$ is strongly balanced. This is true for $k=0$. For $k \ge 1$, it suffices to show that $\nabla S_1[8k]$ is balanced. It will then automatically be strongly balanced since $\nabla S_1[8k-8]$ is, by the induction hypothesis. Thus, we are assuming that $0,1,2,3$ occur with the same multiplicity in $\nabla S_1[8k-8]$, and we must show that this remains true in $\nabla S_1[8k]$.

Consider the multiset difference 
$$
T \; =\; \nabla S_1[8k] \setminus \nabla S_1[8k-8],
$$
a band of width 8 bordering the eastern side of $\nabla S_1[8k]$. To conclude the proof, we need only show that $0,1,2,3$ occur with the same multiplicity in $T$. For any finite multiset $X$ on $\Z{4}$, and for all $j \in \Z{4}$, let us denote by
$$
\m_{X}(j)
$$
the occurrence multiplicity of $j$ in $X$. 

In order to determine the function $\m_{T}$ on $\Z{4}$, we need to determine $\m_{X}$ for the building blocks $X=A_i,B_i,C_i, D_0,E_0$. This is done in Table~\ref{tab1}. 
\begin{table}[!h]
\begin{center}
\begin{tabular}{|c||c|c|c|c|c|c|c|c|c|c|c|c|c|c|}
\hline
 & $A_0$ & $A_1$ & $A_2$ & $A_3$ & $B_0$ & $B_1$ & $B_2$ & $B_3$ & $C_0$ & $C_1$ & $C_2$ & $C_3$ & $D_0$ & $E_0$ \\
\hline\hline
$0$ & $9$ & $5$ & $8$ & $7$ & $20$ & $16$ & $15$ & $16$ & $17$ & $15$ & $17$ & $16$ & $20$ & $19$ \\
\hline
$1$ & $9$ & $10$ & $9$ & $10$ & $15$ & $15$ & $16$ & $17$ & $17$ & $16$ & $15$ & $17$ & $15$ & $14$ \\
\hline
$2$ & $9$ & $11$ & $8$ & $11$ & $14$ & $16$ & $15$ & $14$ & $17$ & $17$ & $15$ & $16$ & $14$ & $17$ \\
\hline
$3$ & $9$ & $10$ & $11$ & $8$ & $15$ & $17$ & $18$ & $17$ & $13$ & $16$ & $17$ & $15$ & $15$ & $14$ \\
\hline
\end{tabular}
\end{center}
\caption{\label{tab1}Multiplicities of $0,1,2,3$ in each building block of $\nabla S_1[8k]$}
\end{table}
Let now $C$ denote the multiset union of $C_1,C_2,C_3$. That is, by definition we have
$$
\m_{C}(j) = \m_{C_1}(j) + \m_{C_2}(j) + \m_{C_3}(j)
$$
for all $j \in \Z{4}$. Looking at the columns below $C_1,C_2,C_3$ in Table~\ref{tab1}, we see that
\begin{equation}\label{m_C}
\m_{C}(j) = 15+16+17 =48
\end{equation}
for all $j \in \Z{4}$. We are now ready to show that $\m_T$ is constant on $\Z{4}$. For this, we need to distinguish 3 cases, according to the class of $k$ mod 3.

\begin{itemize}
\item Case 1: $k=3q$. Figure~\ref{fig1} shows that the building blocks making up $T$ are $A_2$, $B_1$ and $C_0$ occurring once each, and $C_1,C_2,C_3$ occurring $q-1$ times each. Therefore, using Table~\ref{tab1} and \eqref{m_C}, we get
$$
\m_T(j)
\begin{array}[t]{l}
= \m_{A_2}(j) + \m_{B_1}(j) + \m_{C_0}(j) + (q-1) \m_{C}(j)\\[2ex]
= 41 + 48(q-1),
\end{array}
$$
for all $j\in\Z{4}$. 
\item Case 2: $k=3q+1$. In this case, the building blocks making up $T$ are $A_3$, $B_2$, $C_1$ and $D_0$ occurring once each, and $C_1,C_2,C_3$ occurring $q-1$ times each. Thus
$$
\m_T(j)
\begin{array}[t]{l}
= \m_{A_3}(j) +\m_{B_2}(j) + \m_{C_1}(j) + \m_{D_0}(j) + (q-1) \m_{C}(j)\\[2ex]
= 57 + 48(q-1),
\end{array}
$$
for all $j\in\Z{4}$. 
\item Case 3: $k=3q+2$. Now, the building blocks making up $T$ are $A_1$, $B_3$, $C_2$, $C_3$ and $E_0$ occurring once each, and $C_1,C_2,C_3$ occurring $q-1$ times each. Thus
$$
\m_T(j)
\begin{array}[t]{l}
= \m_{A_1}(j) + \m_{B_3}(j) + \m_{C_2}(j) + \m_{C_3}(j) + \m_{E_0}(j) + (q-1) \m_{C}(j)\\[2ex]
= 73 + 48(q-1),
\end{array}
$$
for all $j\in\Z{4}$. 
\end{itemize}
Hence $\m_{T}$ is constant on $\Z{4}$ in each case, as desired. This completes the proof of Theorem~\ref{thm1} for the sequence $S_1$.

The proof for the sequences $S_2,T_1,T_2,T_3,T_4$ follows similar lines. To start with, the structure of each derived triangle is the same as in Figure~\ref{fig1}. Indeed, let $A_0,A_1,A_2,A_3$ be the triangles constructed from the finite subsequences given in Table~\ref{tab2} for each sequence $S_2,T_1,T_2,T_3,T_4$. Let now $B_0,B_1,B_2,B_3,C_0,C_1,C_2,C_3,D_0,E_0$ be the blocks defined by the same formulae \eqref{eq1}, \eqref{eq2} and \eqref{eq3} as in the proof for $S_1$. Then, as easily verified, the equalities \eqref{eq4} still hold. Finally, as for $S_1$, the conclusion of the proof follows from the knowledge of the multiplicities of $0,1,2,3\in \mathbb{Z}/4\mathbb{Z}$ for each block. For convenience, these multiplicities are made explicit in Table~\ref{tab3}.

\begin{table}[!h]
\begin{center}
\begin{tabular}{|c||c|c|c|c|}
\hline
 & $A_0$ & $A_1$ & $A_2$ & $A_3$ \\
\hline\hline
$S_2$ & $\nabla(21210130)$ & $\nabla(20013202)$ & $\nabla(21120021)$ & $\nabla(10220130)$ \\
\hline
$T_1$ & $\nabla(0120021)$ & $\nabla(21220210)$ & $\nabla(20230322)$ & $\nabla(00322021)$ \\
\hline
$T_2$ & $\nabla(1000212)$ & $\nabla(31222330)$ & $\nabla(12103120)$ & $\nabla(03103232)$ \\
\hline
$T_3$ & $\nabla(1200210)$ & $\nabla(22010122)$ & $\nabla(20322221)$ & $\nabla(03000210)$ \\
\hline
$T_4$ & $\nabla(2102203)$ & $\nabla(23200210)$ & $\nabla(20212300)$ & $\nabla(22302203)$ \\
\hline
\end{tabular}
\end{center}
\caption{\label{tab2}Definition of blocks $A_0,A_1,A_2,A_3$ for the sequences $S_2,T_1,T_2,T_3,T_4$.}
\end{table}
\begin{table}[!h]
\vspace{1cm}
\begin{center}
$S_2$ : 
\begin{tabular}{|c||c|c|c|c|c|c|c|c|c|c|c|c|c|c|}
\hline
& $A_0$ & $A_1$ & $A_2$ & $A_3$ & $B_0$ & $B_1$ & $B_2$ & $B_3$ & $C_0$ & $C_1$ & $C_2$ & $C_3$ & $D_0$ & $E_0$ \\
\hline\hline
$0$ & $9$ & $8$ & $6$ & $10$ & $17$ & $19$ & $14$ & $15$ & $16$ & $13$ & $20$ & $15$ & $20$ & $15$ \\
\hline
$1$ & $9$ & $9$ & $10$ & $11$ & $16$ & $11$ & $19$ & $20$ & $20$ & $18$ & $10$ & $20$ & $9$ & $14$ \\
\hline
$2$ & $9$ & $12$ & $12$ & $8$ & $13$ & $13$ & $18$ & $17$ & $16$ & $19$ & $12$ & $17$ & $12$ & $15$ \\
\hline
$3$ & $9$ & $7$ & $8$ & $7$ & $18$ & $21$ & $13$ & $12$ & $12$ & $14$ & $22$ & $12$ & $23$ & $20$ \\
\hline
\end{tabular}\\[2ex] 
$T_1$ : 
\begin{tabular}{|c||c|c|c|c|c|c|c|c|c|c|c|c|c|c|}
\hline
& $A_0$ & $A_1$ & $A_2$ & $A_3$ & $B_0$ & $B_1$ & $B_2$ & $B_3$ & $C_0$ & $C_1$ & $C_2$ & $C_3$ & $D_0$ & $E_0$ \\
\hline\hline
$0$ & $7$ & $5$ & $10$ & $11$ & $18$ & $17$ & $12$ & $20$ & $12$ & $19$ & $17$ & $12$ & $13$ & $17$ \\
\hline
$1$ & $7$ & $10$ & $6$ & $10$ & $13$ & $13$ & $18$ & $12$ & $20$ & $14$ & $15$ & $19$ & $13$ & $15$ \\
\hline
$2$ & $7$ & $13$ & $10$ & $8$ & $10$ & $15$ & $20$ & $13$ & $14$ & $13$ & $16$ & $19$ & $14$ & $10$ \\
\hline
$3$ & $7$ & $8$ & $10$ & $7$ & $15$ & $19$ & $14$ & $19$ & $10$ & $18$ & $16$ & $14$ & $16$ & $14$ \\
\hline
\end{tabular}\\[2ex] 
$T_2$ : 
\begin{tabular}{|c||c|c|c|c|c|c|c|c|c|c|c|c|c|c|}
\hline
& $A_0$ & $A_1$ & $A_2$ & $A_3$ & $B_0$ & $B_1$ & $B_2$ & $B_3$ & $C_0$ & $C_1$ & $C_2$ & $C_3$ & $D_0$ & $E_0$ \\
\hline\hline
$0$ & $7$ & $10$ & $9$ & $10$ & $13$ & $18$ & $16$ & $15$ & $12$ & $15$ & $18$ & $15$ & $14$ & $13$ \\
\hline
$1$ & $7$ & $10$ & $8$ & $10$ & $13$ & $18$ & $16$ & $16$ & $13$ & $15$ & $17$ & $16$ & $14$ & $12$ \\
\hline
$2$ & $7$ & $7$ & $8$ & $8$ & $16$ & $15$ & $16$ & $16$ & $16$ & $16$ & $15$ & $17$ & $15$ & $16$ \\
\hline
$3$ & $7$ & $9$ & $11$ & $8$ & $14$ & $13$ & $16$ & $17$ & $15$ & $18$ & $14$ & $16$ & $13$ & $15$ \\
\hline
\end{tabular}\\[2ex] 
$T_3$ : 
\begin{tabular}{|c||c|c|c|c|c|c|c|c|c|c|c|c|c|c|}
\hline
& $A_0$ & $A_1$ & $A_2$ & $A_3$ & $B_0$ & $B_1$ & $B_2$ & $B_3$ & $C_0$ & $C_1$ & $C_2$ & $C_3$ & $D_0$ & $E_0$ \\
\hline\hline
$0$ & $7$ & $7$ & $10$ & $11$ & $16$ & $15$ & $14$ & $18$ & $14$ & $19$ & $15$ & $14$ & $11$ & $17$ \\
\hline
$1$ & $7$ & $9$ & $8$ & $9$ & $14$ & $16$ & $15$ & $16$ & $15$ & $17$ & $17$ & $14$ & $14$ & $15$ \\
\hline
$2$ & $7$ & $13$ & $9$ & $7$ & $10$ & $17$ & $19$ & $14$ & $13$ & $14$ & $16$ & $18$ & $15$ & $10$ \\
\hline
$3$ & $7$ & $7$ & $9$ & $9$ & $16$ & $16$ & $16$ & $16$ & $14$ & $14$ & $16$ & $18$ & $16$ & $14$ \\
\hline
\end{tabular}\\[2ex] 
$T_4$ : 
\begin{tabular}{|c||c|c|c|c|c|c|c|c|c|c|c|c|c|c|}
\hline
& $A_0$ & $A_1$ & $A_2$ & $A_3$ & $B_0$ & $B_1$ & $B_2$ & $B_3$ & $C_0$ & $C_1$ & $C_2$ & $C_3$ & $D_0$ & $E_0$ \\
\hline\hline
$0$ & $7$ & $8$ & $12$ & $7$ & $15$ & $13$ & $19$ & $18$ & $14$ & $16$ & $15$ & $17$ & $13$ & $13$ \\
\hline
$1$ & $7$ & $9$ & $7$ & $8$ & $14$ & $17$ & $19$ & $13$ & $15$ & $15$ & $14$ & $19$ & $13$ & $16$ \\
\hline
$2$ & $7$ & $10$ & $8$ & $12$ & $13$ & $19$ & $13$ & $15$ & $12$ & $16$ & $18$ & $14$ & $14$ & $14$ \\
\hline
$3$ & $7$ & $9$ & $9$ & $9$ & $14$ & $15$ & $13$ & $18$ & $15$ & $17$ & $17$ & $14$ & $16$ & $13$ \\
\hline
\end{tabular}
\caption{\label{tab3}Multiplicities of $0,1,2,3$ in each building block of $\nabla S_2[8k]$ and $\nabla T_i[8k+7]$ for $i\in\{1,2,3,4\}$.}
\end{center}
\end{table}

\section{The construction method}\label{method}

We now explain how our solution was constructed. Let $m_1,m_2 \ge 2$ be integers, with $m_2$ a multiple of $m_1$. Consider the canonical quotient map
$$
\pi: \Z{m_2} \longrightarrow \Z{m_1}.
$$
If $\nabla$ is a Steinhaus triangle in $\Z{m_2}$, then $\pi(\nabla)$ is a Steinhaus triangle in $\Z{m_1}$. Moreover, if $\nabla$ is balanced, then so is $\pi(\nabla)$, as all fibers of $\pi$ have the same cardinality. Thus, an obvious strategy for constructing balanced Steinhaus triangles in $\Z{m_2}$ consists in trying to lift to $\Z{m_2}$ known balanced Steinhaus triangles in $\Z{m_1}$. This route is tricky, as illustrated by Theorems~\ref{thm3} and \ref{thm4} below. It allowed us to solve the case $m=4$ of Molluzzo's problem, but neither the case $m=6$ nor the case $m=8$ so far.

We shall restrict our attention to \textit{strongly balanced} Steinhaus triangles. These were defined earlier in $\Z{4}$ only. We now generalize them to $\Z{m}$ for all even moduli.

\begin{defn}\label{strong m}
Let $m \ge 2$ be an even modulus. Let $S$ be a finite sequence of length $n\ge 0$ in $\Z{m}$. The Steinhaus triangle $\nabla S$ is said to be \textit{strongly balanced} if, for every $0\le t \le n/(2m)$, the Steinhaus triangle $\nabla S[n-2mt]$ is balanced.
\end{defn}

Note that this definition coincides with Definition~\ref{strong} for $m=4$. From now on, we assume that $m_1=m$ is an even number, and that $m_2=2m_1$. The following notation helps to measure, roughly speaking, to what extent strong solutions in $\Z{m}$ can be lifted to strong solutions in $\Z{2m}$.

\medskip

\begin{notation}
Let $S$ be an infinite sequence in $\Z{m}$. For $n \ge 0$, let $a_n(S)$ denote the number of sequences $T$ in $\Z{2m}$, of length $n$, such that 
\begin{itemize}
\item $\nabla T$ is a strongly balanced Steinhaus triangle in $\Z{2m}$;
\item $\pi(T)=S[n]$, the initial segment of length $n$ in $S$.
\end{itemize}
We denote by $\mathcal{G}_S(t)=\sum_{n=0}^{\infty}a_{n}(S)t^{n}$ the generating function of the numbers $a_{n}(S)$.
\end{notation}

\medskip

We shall use this notation as a convenient device for exhibiting the value of the $a_{n}(S)$ for all $n$ at once. For our present purposes, the favorable case occurs when $\mathcal{G}_S(t)$ is an infinite series, not just a polynomial. Indeed, $\mathcal{G}_S(t)$ is an infinite series if and only if there exists infinitely many strongly balanced Steinhaus triangles in $\Z{2m}$, which lift those in $\Z{m}$ generated by initial segments of $S$.

\subsection{From $\Z{2}$ to $\Z{4}$} Here we set $m=2$. Several types of balanced Steinhaus triangles of length $4k$ or $4k+3$ in $\Z{2}$ are known. See \cite{Ha,EH,EMR}. We focus here on the ones given in \cite{EH}, which have the added property of being strongly balanced.

\begin{thm}[\cite{EH}]\label{thm2}
Let $Q_1,\ldots,Q_4$ and $R_1,\ldots,R_{12}$ be the following eventually periodic sequences of $\Z{2}$:
\begin{center}
\begin{tabular}{cc}
$
\begin{array}{l}
Q_1 = 0100(001001011100)^{\infty}, \\
Q_2 = (010010000111)^{\infty}, \\
Q_3 = 0101(011000011000)^{\infty}, \\
Q_4 = 0101(101000101000)^{\infty}, \\[4ex]
R_1 = 001(010000100001)^{\infty}, \\
R_2 = 0011110(001101010110)^{\infty}, \\
R_3 = 010(000101000010)^{\infty}, \\
\end{array}
$
&
$
\begin{array}{l}
R_4 = 0100001(010010111100001010111111)^{\infty}, \\
R_5 = 0100001(100100001001)^{\infty}, \\
R_6 = 0101011(010101100011)^{\infty}, \\
R_7 = 0101011(010111111101011010011101)^{\infty}, \\
R_8 = 010(101110110010)^{\infty}, \\
R_9 = 100(001000010100)^{\infty}, \\
R_{10} = 1000010(110001101010)^{\infty}, \\
R_{11} = 1111101(011000110101)^{\infty}, \\
R_{12} = 111(110110000111)^{\infty}.
\end{array}
$\\
\end{tabular}
\end{center}
For all integers $i,j,k$ such that $1 \le i \le 4$, $1 \le j \le 12$ and $k \ge 0$, the Steinhaus triangles $\nabla Q_i[4k]$ and $\nabla R_j[4k+3]$ are strongly balanced in $\Z{2}$.
\end{thm}

Can we lift some initial segments of these sequences to sequences in $\Z{4}$ which generate strongly balanced Steinhaus triangles? 
To answer this question, we have determined by computer the numbers $a_n(S)$ for all 16 sequences $S$ in Theorem~\ref{thm2} and all $n \ge 1$. In 11 out of the 16 cases, the numbers $a_n(S)$ turn out to vanish for all sufficiently large $n$, i.e. the series $\mathcal{G}_S(t)$ is just a polynomial. But remarkably, in the remaining 5 cases, the $a_n(S)$ turn out to be \textit{ultimately periodic} and non-vanishing, so that the infinite series $\mathcal{G}_S(t)$ is actually a \textit{rational function}. These 16 series are displayed below; the 5 infinite ones occur for the sequences $Q_1,Q_3,R_3,R_9,R_{10}$.

\begin{thm}\label{thm3}
The generating functions $\mathcal{G}_S(t)$ of $Q_1,\ldots,Q_4$ and $R_1,\ldots,R_{12}$ are:
\end{thm}
\vspace{-1cm}
$$
\begin{array}{l}
\mathcal{G}_{Q_1}(t) = \begin{array}[t]{l} 1 + 8t^{8} + 34t^{16} + 58t^{24} + 84t^{32} + 88t^{40} + 86t^{48} + 82t^{56} + 60t^{64} + 36t^{72} + \\ \displaystyle + 34t^{80} + 28t^{88} + 16t^{96} + \frac{2t^{104}}{1-t^{8}},\end{array}\\
\ \\
\mathcal{G}_{Q_2}(t) = \begin{array}[t]{l} 1 + 4t^{8} + 14t^{16} + 32t^{24} + 36t^{32} + 48t^{40} + 44t^{48} + 26t^{56} + 22t^{64} + 8t^{72} + \\ \displaystyle + 6t^{80} + 4t^{88} + 2t^{96},\end{array}\\
\ \\
\mathcal{G}_{Q_3}(t) = \begin{array}[t]{l}
1 + 8t^{8} + 28t^{16} + 46t^{24} + 78t^{32} + 124t^{40} + 118t^{48} + 96t^{56} + 78t^{64} + 60t^{72} + \\
28t^{80} + 20t^{88} + 14t^{96} + 10t^{104} + 4t^{112} + 6t^{120} + 4t^{128} + 6t^{136} + 4t^{144} + 2t^{152} + \\ \displaystyle + 2t^{160} + 2t^{168} + 2t^{176} + 2t^{184} + 2t^{192} + 2t^{200} + 4t^{208} + \frac{2t^{216}}{1-t^{8}},\end{array}\\
\ \\
\mathcal{G}_{Q_4}(t) = 1 + 8t^{8} + 26t^{16} + 42t^{24} + 66t^{32} + 62t^{40} + 52t^{48} + 36t^{56} + 26t^{64} + 12t^{72} + 6t^{80},\\
\ \\
\mathcal{G}_{R_1}(t) = 0,\quad \mathcal{G}_{R_2}(t)=0,\\
\ \\
\mathcal{G}_{R_3}(t) = \begin{array}[t]{l}
10t^7 + 38t^{15} + 70t^{23} + 88t^{31} + 76t^{39} + 54t^{47} + 44t^{55} + 28t^{63} + 16t^{71} + 8t^{79} +\\
\displaystyle + 4t^{87} + 4t^{95} + 4t^{103} + 4t^{111} + 4t^{119} + 6t^{127} + 4t^{135} + 6t^{143} + \frac{4t^{151}}{1-t^{8}},
\end{array}\\
\ \\
\mathcal{G}_{R_4}(t) = \begin{array}[t]{l}
10t^7 + 52t^{15} + 102t^{23} + 136t^{31} + 152t^{39} + 118t^{47} + 108t^{55} + 80t^{63} + 60t^{71} + \\
\displaystyle + 32t^{79} + 20t^{87} + 8t^{95} + 2t^{103},
\end{array}\\
\ \\
\mathcal{G}_{R_5}(t) = 10t^7,
\end{array}
$$

$$
\begin{array}{l}

\mathcal{G}_{R_6}(t) = \begin{array}[t]{l}
10t^7 + 30t^{15} + 66t^{23} + 96t^{31} + 96t^{39} + 94t^{47} + 66t^{55} + 42t^{63} + 24t^{71} + 8t^{79} +\\
\displaystyle + 2t^{87} + 2t^{95},
\end{array}\\
\ \\
\mathcal{G}_{R_7}(t) = \begin{array}[t]{l}
10t^7 + 60t^{15} + 138t^{23} + 204t^{31} + 304t^{39} + 266t^{47} + 246t^{55} + 148t^{63} + 64t^{71} + \\
\displaystyle + 36t^{79} + 14t^{87} + 10t^{95} + 8t^{103},
\end{array}\\
\ \\
\mathcal{G}_{R_8}(t) = 10t^7,\\
\ \\
\mathcal{G}_{R_9}(t) = \begin{array}[t]{l}
10t^7 + 42t^{15} + 80t^{23} + 130t^{31} + 164t^{39} + 174t^{47} + 126t^{55} + 68t^{63} + 38t^{71} + \\
\displaystyle + 20t^{79} + 22t^{87} + 12t^{95} + 2t^{103} + 2t^{111} + 2t^{119} + 2t^{127} + 2t^{135} + 2t^{143} + 2t^{151} + \\
\displaystyle + 2t^{159} + 2t^{167} + 2t^{175} + 2t^{183} + 2t^{191} + 2t^{199} + 4t^{207} + \frac{2t^{215}}{1-t^8},
\end{array}\\
\ \\
\mathcal{G}_{R_{10}}(t) = \begin{array}[t]{l}
10t^7 + 58t^{15} + 98t^{23} + 130t^{31} + 160t^{39} + 138t^{47} + 132t^{55} + 84t^{63} + 64t^{71} + \\
\displaystyle + 34t^{79} + 14t^{87} + 8t^{95} + 6t^{103} + 2t^{111} + 2t^{119} + 4t^{127} + \frac{2t^{135}}{1-t^8},
\end{array}\\
\ \\
\mathcal{G}_{R_{11}}(t) = 4t^7 + 16t^{15} + 26t^{23} + 32t^{31} + 30t^{39} + 30t^{47} + 26t^{55} + 12t^{63} + 8t^{71} + 2t^{79},\\
\ \\
\mathcal{G}_{R_{12}}(t)= 4t^{7}.
\end{array}
$$

The origin of our sequences $S_1$, $S_2$, $T_1$, $T_2$, $T_3$, $T_4$, solving the problem of Molluzzo in $\Z{4}$, is now clear. Indeed, they are lifts to $\Z{4}$ of the 5 sequences $Q_1,Q_3,R_3,R_9,R_{10}$ in $\Z{2}$ with $\mathcal{G}_S(t)$ infinite. More precisely, we have 
$$
\pi(S_1)=Q_1,\;\; \pi(S_2)=Q_3,\;\;  \pi(T_1)=\pi(T_4)=R_3,\;\;  \pi(T_2)=R_{10},\;\;  \pi(T_3)=R_{9},$$ 
as the reader may readily check.

\subsection{From $\Z{4}$ to $\Z{8}$} 

Having solved the problem in $\Z{4}$ with Theorem~\ref{solution}, can we lift our solutions $S_1$, $S_2$, $T_1$, $T_2$, $T_3$, $T_4$ to sequences in $\Z{8}$ giving rise to infinitely many strongly balanced Steinhaus triangles? Unfortunately, the answer is no, as shown by the following computational result.

\begin{thm}\label{thm4}
The generating functions $\mathcal{G}_S(t)$ of $S_1,S_2,T_1,T_2,T_3,T_4$ are polynomials only:
$$
\begin{array}{l}
\mathcal{G}_{S_1}(t) = 1 + 16t^{16} + 46t^{32} + 32t^{48} + 14t^{64},\\
\ \\
\mathcal{G}_{S_2}(t) = 1 + 22t^{16} + 60t^{32} + 56t^{48} + 28t^{64} + 6t^{80},\\
\ \\
\mathcal{G}_{T_1}(t) = 14t^{15} + 40t^{31} + 40t^{47} + 24t^{63} + 8t^{79} + 2t^{95} + 2t^{111},\\
\ \\
\mathcal{G}_{T_2}(t) = 30t^{15} + 66t^{31} + 76t^{47} + 32t^{63} + 12t^{79},\\
\ \\
\mathcal{G}_{T_3}(t) = 14t^{15} + 54t^{31} + 42t^{47} + 34t^{63} + 12t^{79} + 2t^{95},\\
\ \\
\mathcal{G}_{T_4}(t) = 14t^{15} + 54t^{31} + 64t^{47} + 40t^{63} + 10t^{79} + 2t^{95}.
\end{array}
$$
\end{thm}

Summarizing, at this stage, it is not even known whether there exist infinitely many balanced Steinhaus triangles in $\Z{8}$.

\section{A weaker version of the problem}\label{weaker}

With our present solution of the case $m=4$, the currently smallest open case of Molluzzo's problem becomes $m=6$. The scarcity of solved instances $(m=2,3^k,4,5,7)$ motivates us to propose a weaker, more accessible version of the problem. 

\begin{problem}[The Weak Molluzzo Problem] Let $m \in \N$, $m \ge 2$. Are there infinitely many balanced Steinhaus triangles in $\Z{m}$?
\end{problem}

The picture is brighter here. Indeed, the first author has shown in \cite{Ch0, Ch1, Ch2} that, \textit{for each odd modulus $m$, there are infinitely many balanced Steinhaus triangles in $\Z{m}$}. Thus, the weak Molluzzo problem is affirmatively solved for all odd $m$, for $m=2$ and here for $m=4$. On the other hand, it is widely open for all even moduli $m \ge 6$. 

Somewhat similarly to the conjecture on the existence of Hadamard matrices of every order divisible by 4, the nature of the problem seems to lie less in the rarity of the solutions than in the difficulty of pinpointing easy-to-describe ones. For instance, in $\Z{6}$, there are exactly 94648 sequences of length 12 yielding a balanced Steinhaus triangle; up to automorphisms, they still total 23662 classes.

\medskip

We note, finally, that all known solutions so far are by explicit constructions. However, the possibility of future nonconstructive existence results cannot be ruled out, for instance with the polynomial method of Alon-Tarsi.

\end{document}